\documentclass[12pt]{article}

\usepackage{amsmath, epsfig, cite, setspace}
\usepackage{amssymb}
\usepackage{amsfonts}
\usepackage{latexsym}
\usepackage{amsthm}

\usepackage{soul}

\makeatletter
\renewcommand{\@seccntformat}[1]{{\csname the#1\endcsname}.\hspace{.5em}}
\makeatother

\newtheorem{thm}{Theorem}[section]

\newtheorem{conj}[thm]{Conjecture}
\newtheorem{lem}[thm]{Lemma}

\renewcommand{\thefootnote}{*}

\numberwithin{equation}{section}

\begin{document}

\begin{center}
{\large\bf $q$-Supercongruences modulo the fourth power of a cyclotomic\\[2pt] polynomial via creative microscoping}
\end{center}

\vskip 2mm \centerline{Victor J. W. Guo}
\begin{center}
{\footnotesize School of Mathematics and Statistics, Huaiyin Normal
University, Huai'an 223300, Jiangsu, People's Republic of China\\
{\tt jwguo@hytc.edu.cn }  }
\end{center}


\vskip 0.7cm \noindent{\bf Abstract.} By applying Chinese remainder theorem for coprime polynomials
and the ``creative microscoping" method recently introduced by the author and Zudilin,
we establish parametric generalizations of three $q$-supercongruences modulo the fourth power of a cyclotomic polynomial.
The original $q$-supercongruences then follow from these parametric generalizations by taking the limits as
the parameter tends to $1$ (l'H\^opital's rule is utilized here). In particular, we prove a complete $q$-analogue of the (J.2) supercongruence of Van Hamme
and a complete $q$-analogue of a  ``divergent" Ramanujan-type supercongruence, thus confirming two recent conjectures of the author.
We also put forward some related conjectures, including a $q$-supercongruence modulo the fifth power of a cyclotomic polynomial.


\renewcommand{\thefootnote}{**}

\section{Introduction}
In his second notebook, Ramanujan mysteriously recorded 17 infinite series representations of $1/\pi$ дл (see \cite[p. 352]{Berndt}),
including for instance
\begin{equation*}
\sum_{k=0}^\infty (6k+1)\frac{(\frac{1}{2})_k^3}{k!^3 4^k}
=\frac{4}{\pi},
\end{equation*}
which he later published in \cite{Ramanujan}. Here and throughout the paper, $(a)_n=a(a+1)\cdots(a+n-1)$ stands for the Pochhammer symbol.
It was noticed by Van Hamme \cite{Hamme} in 1997 that several Ramanujan's and Ramanujan-type formulas possess nice $p$-adic analogues, such as
\begin{align}
\sum_{k=0}^{(p-1)/2} (4k+1)\frac{(\frac{1}{2})_k^4}{k!^4}
&\equiv p\pmod{p^3},  \label{eq:c2} \\[5pt]
\sum_{k=0}^{(p-1)/2} (6k+1)\frac{(\frac{1}{2})_k^3}{k!^3 4^k}
&\equiv (-1)^{(p-1)/2}p\pmod{p^4},  \label{eq:j2}
\end{align}
where $p>3$ is a prime.
The supercongruence \eqref{eq:c2} was proved by Van Hamme \cite[(C.2)]{Hamme} himself. Later Long \cite{Long}
proved that both \eqref{eq:c2} and \eqref{eq:j2} are true modulo $p^4$.
It was not until 2016 that Van Hamme's last supercongruence was confirmed by Osburn and Zudilin \cite{OZ} using the WZ
method \cite{WZ}.
For more Ramanujan-type supercongruences, see Zudilin's famous paper \cite{Zud2009}.

During the past few years, many congruences and supercongruences have been generalized to the $q$-settings by different authors (see, for example,
\cite{Gorodetsky,Guo2018,Guo-J,Guo-m3,Guo-div,GL18,GPZ,GS,GS2,GS3,GW,GuoZu,NP,Straub,Tauraso2}).
In particular, using the $q$-WZ method \cite{WZ} the author and Wang \cite{GW} established a $q$-analogue of \eqref{eq:c2}: for odd $n$,
\begin{align}
\sum_{k=0}^{(n-1)/2}[4k+1]\frac{(q;q^2)_k^4}{(q^2;q^2)_k^4}
\equiv q^{(1-n)/2}[n]+\frac{(n^2-1)(1-q)^2}{24}q^{(1-n)/2}[n]^3 \pmod{[n]\Phi_n(q)^3}.  \label{eq:gw}
\end{align}
They \cite[Conjecture 5.1]{GW} also asserted that the above $q$-congruence is also true when the sum on the left-hand side is over $k$ from 0 to $n-1$.
Here and in what follows we adopt the standard $q$-hypergeometric notation:
$(a;q)_n=(1-a)(1-aq)\cdots (1-aq^{n-1})$ is the {\em $q$-shifted factorial};
$[n]=[n]_q=1+q+\cdots+q^{n-1}$ is the {\em $q$-integer};
and $\Phi_n(q)$ denotes the $n$-th {\em cyclotomic polynomial} in $q$:
\begin{align*}
\Phi_n(q)=\prod_{\substack{1\leqslant k\leqslant n\\ \gcd(n,k)=1}}(q-\zeta^k),
\end{align*}
where $\zeta$ is an $n$-th primitive root of unity.

Moreover, the author and Zudilin \cite{GuoZu} devised a method, called ``creative microscoping",
to prove many $q$-supercongruences modulo $\Phi_n(q)^3$
by adding one or more extra parameters and considering asymptotics at roots of unity.
Later the author and Schlosser \cite{GS} applied the creative microscoping method
to deduce many other $q$-supercongruences from transformation formulas for basic
However, no $q$-supercongruences modulo $\Phi_n(q)^4$ are proved by the creative microscoping method
up to now.

In this paper, we shall give a creative microscoping proof of \eqref{eq:gw}. More precisely,
we shall establish the following parametric generalization of \eqref{eq:gw}.

\begin{thm}\label{thm:main-1}
Let $n$ be a positive odd integer. Then, modulo $[n]\Phi_n(q)(1-aq^n)(a-q^n)$,
\begin{align}
&\sum_{k=0}^{(n-1)/d}[4k+1]\frac{(aq;q^2)_k(q/a;q^2)_k(q;q^2)_k^2}{(aq^2;q^2)_k(q^2/a;q^2)_k(q^2;q^2)_k^2}  \notag\\
&\quad\equiv q^{(1-n)/2}[n]+q^{(1-n)/2}[n]\frac{(1-aq^n)(a-q^n)}{(1-a)^2}\left(1-\frac{n(1-a)a^{(n-1)/2}}{1-a^n}\right),  \label{eq:main-1}
\end{align}
where $d=1,2$.
\end{thm}

By l'H\^opital's rule, we have
\begin{align*}
\lim_{a\to 1}\frac{(1-aq^n)(a-q^n)}{(1-a)^2}\frac{(1-a^n-n(1-a)a^{(n-1)/2})}{(1-a^n)}
=\frac{(n^2-1)(1-q)^2}{24}[n]^2.
\end{align*}
Thus, taking the limits of the two sides of \eqref{eq:main-1} as $a\to 1$, we see that \eqref{eq:gw}
is true modulo $\Phi_n(q)^4$. But the proof of \cite[Theorem 12.9]{GS}
(or \cite[Theorem 4.2]{GuoZu}) already indicates that it is also true modulo $[n]$.
Therefore, the $q$-congruence \eqref{eq:gw} and \cite[Conjecture 5.1]{GW} are consequences of \eqref{eq:main-1}.

Likewise, there is a $q$-analogue of \eqref{eq:j2} (i.e., the (J.2) supercongruence of Van Hamme \cite{Hamme}) proposed by
the author \cite[Conjecture 1.1]{Guo-J}: for odd $n$,
\begin{align}
&\sum_{k=0}^{(n-1)/2}q^{k^2}[6k+1]\frac{(q;q^2)_k^2 (q^2;q^4)_k }{(q^4;q^4)_k^3} \notag\\[5pt]
&\quad\equiv (-q)^{(1-n)/2}[n]+\frac{(n^2-1)(1-q)^2}{24}(-q)^{(1-n)/2}[n]^3 \pmod{[n]\Phi_n(q)^3}.  \label{eq:qj2}
\end{align}
The author \cite{Guo-J} proved that \eqref{eq:qj2} is true modulo $[n]\Phi_n(q)$,
and the author and Zudilin \cite{GuoZu} proved that \eqref{eq:qj2} is also true modulo $[n]\Phi_n(q)^2$
by the aforementioned method of creative microscoping.
Here we shall completely confirm \eqref{eq:qj2} by showing the following parametric generalization.
\begin{thm}\label{thm:main-2}
Let $n$ be a positive odd integer. Then, modulo $[n]\Phi_n(q)(1-aq^n)(a-q^n)$,
\begin{align}
&\sum_{k=0}^{(n-1)/2}q^{k^2}[6k+1]\frac{(aq;q^2)_k(q/a;q^2)_k (q^2;q^4)_k }{(aq^4;q^4)_k(q^4/a;q^4)_k(q^4;q^4)_k} \notag\\
&\quad\equiv (-q)^{(1-n)/2}[n]+(-q)^{(1-n)/2}[n]\frac{(1-aq^n)(a-q^n)}{(1-a)^2}\left(1-\frac{n(1-a)a^{(n-1)/2}}{1-a^n}\right).  \label{eq:main-2}
\end{align}
\end{thm}

It is clear that the $q$-congruence \eqref{eq:qj2} modulo $\Phi_n(q)^4$ follows from \eqref{eq:main-2} by taking the limit as $a\to 1$.
Since \eqref{eq:qj2} modulo $[n]$ has already been given in \cite{Guo-J} and \cite{GuoZu}, the $q$-congruence \eqref{eq:qj2} is thus a direct
conclusion of \eqref{eq:main-2}.

Partially motivated by a supercongruence of Sun \cite[Conjecture 5.1(ii)]{Sun}, the author \cite[Conjecture 7.1]{Guo-div} made the following conjecture:
for odd $n$,
\begin{align}
&\sum_{k=0}^{n-1}[3k+1]\frac{(q;q^2)_k^3 q^{-{k+1\choose 2} } }{(q;q)_k^2 (q^2;q^2)_k} \notag\\[5pt]
&\qquad
\equiv q^{(1-n)/2}[n]+\frac{(n^2-1)(1-q)^2}{24}q^{(1-n)/2}[n]^3 \pmod{[n]\Phi_n(q)^3}.  \label{eq:q-div}
\end{align}
The $q$-congruence \eqref{eq:q-div} modulo $[n]\Phi_n(q)^2$ was proved by the author \cite{Guo-div} himself,
and was reproved by the author and Zudilin \cite{GuoZu} by establishing its parametric generalization.
Here we shall prove that \eqref{eq:q-div} is true by the method of creative microscoping again.
Namely, we shall establish the following $q$-congruence.
\begin{thm}\label{thm:main-3}
Let $n$ be a positive odd integer. Then, modulo $[n]\Phi_n(q)(1-aq^n)(a-q^n)$,
\begin{align}
&\sum_{k=0}^{n-1}[3k+1]\frac{(aq;q^2)_k(q/a;q^2)_k(q;q^2)_k  q^{-{k+1\choose 2} } }{(aq;q)_k(q/a;q)_k (q^2;q^2)_k} \notag\\[5pt]
&\quad\equiv q^{(1-n)/2}[n]+q^{(1-n)/2}[n]\frac{(1-aq^n)(a-q^n)}{(1-a)^2}\left(1-\frac{n(1-a)a^{(n-1)/2}}{1-a^n}\right).  \label{eq:main-3}
\end{align}
\end{thm}

As before, we can deduce that \eqref{eq:q-div} is true modulo $\Phi_n(q)^4$ (and is therefore also true modulo $[n]\Phi_n(q)^3$)
from  \eqref{eq:main-3} by taking $a\to 1$.

The rest of the paper is organized as follows. We shall prove Theorems \ref{thm:main-1}--\ref{thm:main-3} in Sections 2--4, respectively.
More concretely, we shall establish first the corresponding $q$-congruences modulo $[n](1-aq^n)(a-q^n)(b-q^n)$,
where Chinese remainder theorem is utilized. Then Theorems \ref{thm:main-1}--\ref{thm:main-3} immediately follow from these $q$-congruences with
parameters $a$ and $b$ by letting $b\to 1$. Finally, we propose some related conjectures for further study in Section 5.

\section{Proof of Theorem \ref{thm:main-1}}

We first require the following easily proved lemma.
\begin{lem}
Let $n$ be a positive odd integer. Then
\begin{align}
(aq^2,q^2)_{(n-1)/2}(q^2/a,q^2)_{(n-1)/2}
&\equiv (-1)^{(n-1)/2}\frac{(1-a^n)q^{-(n-1)^2/4}}{(1-a)a^{(n-1)/2}} \pmod{\Phi_n(q)}, \label{eq:lem-1} \\[5pt]
(aq,q^2)_{(n-1)/2}(q/a,q^2)_{(n-1)/2}
&\equiv  (-1)^{(n-1)/2}\frac{(1-a^n)q^{(1-n^2)/4}}{(1-a)a^{(n-1)/2}} \pmod{\Phi_n(q)}.  \label{eq:lem-2}
\end{align}
\end{lem}
\begin{proof}It is easy to see that
\begin{align*}
(q^2/a,q^2)_{(n-1)/2}
&=(1-q^2/a)(1-q^4/a)\cdots(1-q^{n-1}/a) \\
&\equiv (1-q^{2-n}/a)(1-q^{4-n}/a)\cdots (1-q^{-1}/a)\\
&=(-1)^{(n-1)/2}(aq,q^2)_{(n-1)/2}\frac{q^{-(n-1)^2/4}}{a^{(n-1)/2}} \pmod{\Phi_n(q)}.
\end{align*}
Therefore, the left-hand side of \eqref{eq:lem-1} is congruent to
\begin{align*}
(-1)^{(n-1)/2}(aq;q)_{n-1}\frac{q^{-(n-1)^2/4}}{a^{(n-1)/2}}.
\end{align*}
The proof of \eqref{eq:lem-1} then follows from the fact
\begin{align}
(aq;q)_{n-1}=\sum_{k=0}^{n-1}(-1)^k q^{k+1\choose 2}{n-1\brack k}a^k
\equiv \sum_{k=0}^{n-1}a^k \pmod{\Phi_n(q)}, \label{eq:q-bino-a}
\end{align}
where we have used the $q$-binomial theorem in the above equality.

Similarly, we can prove \eqref{eq:lem-2}.
\end{proof}

Recall that the author and Zudilin \cite[Theorem 4.2]{GuoZu} proved the following $q$-congruence:
\begin{lem}
\label{th:4.2}
Let $n$ be a positive odd integer. Then, modulo $[n](1-aq^n)(a-q^n)$,
\begin{equation}
\sum_{k=0}^{(n-1)/d}[4k+1]\frac{(aq;q^2)_k (q/a;q^2)_k (q/b;q^2)_k (q;q^2)_k}
{(aq^2;q^2)_k(q^2/a;q^2)_k (bq^2;q^2)_k (q^2;q^2)_k}b^k
\equiv\frac{(b/q)^{(n-1)/2} (q^2/b;q^2)_{(n-1)/2}}{(bq^2;q^2)_{(n-1)/2}}[n],
\label{eq:q-Long-gen}
\end{equation}
where $d=1,2$.
\end{lem}

We find that the left-hand side of \eqref{eq:q-Long-gen} also has a simple expression modulo $b-q^n$.
\begin{lem}\label{lem:2}
Let $n$ be a positive odd integer. Then, modulo $b-q^n$,
\begin{equation}
\sum_{k=0}^{(n-1)/d}[4k+1]\frac{(aq;q^2)_k (q/a;q^2)_k (q/b;q^2)_k (q;q^2)_k}
{(aq^2;q^2)_k(q^2/a;q^2)_k (bq^2;q^2)_k (q^2;q^2)_k}b^k
\equiv\frac{(q;q^2)_{(n-1)/2}^2 [n]}{(aq^2;q^2)_{(n-1)/2}(q^2/a;q^2)_{(n-1)/2}},
\label{eq:q-Long-gen-2}
\end{equation}
where $d=1,2$.
\end{lem}
\begin{proof}
Note that Jackson's ${}_6\phi_5$ summation formula can be written as
\begin{equation*}
\sum_{k=0}^N\frac{(1-aq^{2k})(a;q)_k(b;q)_k(c;q)_k(q^{-N};q)_k}
{(1-a)(q;q)_k(aq/b;q)_k(aq/c;q)_k(aq^{N+1};q)_k}\biggl(\frac{aq^{N+1}}{bc}\biggr)^k
=\frac{(aq;q)_N (aq/bc;q)_N}{(aq/b;q)_N(aq/c;q)_N}
\end{equation*}
(see \cite[Appendix (II.21)]{GR}). Letting $q\mapsto q^2$ and taking $a=q$, $b=aq$ $c=q/a$ and $N=(n-1)/2$
in the above formula, we obtain
\begin{equation*}
\sum_{k=0}^{(n-1)/2}[4k+1]\frac{(aq;q^2)_k (q/a;q^2)_k (q^{1-n};q^2)_k (q;q^2)_k}
{(aq^2;q^2)_k(q^2/a;q^2)_k (q^{n+2};q^2)_k (q^2;q^2)_k}q^{nk}
=\frac{(q^3;q^2)_{(n-1)/2}(q;q^2)_{(n-1)/2}}{(aq^2;q^2)_{(n-1)/2}(q^2/a;q^2)_{(n-1)/2}}.
\end{equation*}
Namely, when $b=q^n$ both sides of \eqref{eq:q-Long-gen-2} are equal, thus establishing this $q$-congruence.
\end{proof}

\begin{proof}[Proof of Theorem {\rm\ref{thm:main-1}}]It is clear that the polynomials $[n](1-aq^n)(a-q^n)$ and $b-q^n$ are relatively prime.
By Chinese reminder theorem for relatively prime polynomials, we can determine the remainder of the left-hand side of \eqref{eq:q-Long-gen} modulo $[n](1-aq^n)(a-q^n)(b-q^n)$
from \eqref{eq:q-Long-gen} and \eqref{eq:q-Long-gen-2}.
To accomplish this, we need the following $q$-congruences:
\begin{align}
\frac{(b-q^n)(ab-1-a^2+aq^n)}{(a-b)(1-ab)}&\equiv 1\pmod{(1-aq^n)(a-q^n)}, \label{eq:ab-1} \\[5pt]
\frac{(1-aq^n)(a-q^n)}{(a-b)(1-ab)}&\equiv 1\pmod{b-q^n}.    \label{eq:ab-2}
\end{align}
Therefore, combining \eqref{eq:q-Long-gen} and \eqref{eq:q-Long-gen-2} we obtain
\begin{align}
&\sum_{k=0}^{(n-1)/d}[4k+1]\frac{(aq;q^2)_k (q/a;q^2)_k (q/b;q^2)_k (q;q^2)_k}
{(aq^2;q^2)_k(q^2/a;q^2)_k (bq^2;q^2)_k (q^2;q^2)_k}b^k \notag\\[5pt]
&\quad\equiv\frac{(b/q)^{(n-1)/2} (q^2/b;q^2)_{(n-1)/2}}{(bq^2;q^2)_{(n-1)/2}}\frac{(b-q^n)(ab-1-a^2+aq^n)}{(a-b)(1-ab)}[n] \notag\\[5pt]
&\quad\quad+\frac{(q;q^2)_{(n-1)/2}^2 }{(aq^2;q^2)_{(n-1)/2}(q^2/a;q^2)_{(n-1)/2}}
\frac{(1-aq^n)(a-q^n)}{(a-b)(1-ab)}[n] \label{eq:abqn}
\end{align}
modulo $[n](1-aq^n)(a-q^n)(b-q^n)$.

Moreover, by \eqref{eq:lem-1} and \eqref{eq:lem-2} we have
\begin{align}
\frac{(q;q^2)_{(n-1)/2}^2}{(aq,q^2)_{(n-1)/2}(q/a,q^2)_{(n-1)/2}}
\equiv \frac{n(1-a)a^{(n-1)/2}}{(1-a^n)q^{(n-1)/2}}  \pmod{\Phi_n(q)}.  \label{eq:an-frac}
\end{align}
It is easy to see that the limit of $b-q^n$ as $b\to 1$ has the factor $\Phi_n(q)$. Meanwhile, the factor $(bq^2;q^2)_{(n-1)/d}$
in the denominator of the left-hand side of \eqref{eq:abqn} as $b\to 1$ is relatively prime to $\Phi_n(q)$.
Thus, letting $b\to 1$ in \eqref{eq:abqn} and applying \eqref{eq:an-frac}, we conclude that \eqref{eq:main-1}
is true modulo $\Phi_n(q)^2(1-aq^n)(a-q^n)$. Here we used the following relation:
\begin{equation}
(1-q^n)(1+a^2-a-aq^n)=(1-a)^2+(1-aq^n)(a-q^n).  \label{eq:relation}
\end{equation}
Note that \eqref{eq:q-Long-gen} is also true for $b=1$.
That is, the $q$-congruence \eqref{eq:main-1} is true modulo $[n]$. The proof then follows from
the fact that the least common multiple of $\Phi_n(q)^2(1-aq^n)(a-q^n)$ and $[n]$ is $[n]\Phi_n(q)(1-aq^n)(a-q^n)$.
\end{proof}

\section{Proof of Theorem \ref{thm:main-2}}
Similarly as before, we need the following lemma, which was proved by the author and Zudilin \cite[Theorem~4.5]{GuoZu}.
\begin{lem}
\label{thm:q-Hamme-J2L2}
Let $n\equiv r\pmod4$ be a positive odd integer, where $r=\pm1$. Then, modulo $[n](1-aq^n)(a-q^n)$,
\begin{align}
&\sum_{k=0}^{(n-1)/2}[6k+1]\frac{(aq;q^2)_k (q/a;q^2)_k (q;q^2)_k (q^2/b;q^4)_k b^kq^{k^2}}
{(aq^4;q^4)_k(q^4/a;q^4)_k(q^4;q^4)_k (bq;q^2)_k}  \notag\\
&\qquad
\equiv \frac{(q^{2+r}/b;q^4)_{(n-r)/4}}{(bq^{2+r};q^4)_{(n-r)/4}} b^{(n-r)/4}(-q)^{(1-n)/2} [n].  \label{eq:q-Hamme-J2L2}
\end{align}
\end{lem}

We find that the left-hand side of \eqref{eq:q-Hamme-J2L2} has a simple formula modulo $b-q^{2n}$.
\begin{lem}
Let $n$ be a positive odd integer. Then, modulo $b-q^{2n}$,
\begin{align}
&\sum_{k=0}^{(n-1)/2}[6k+1]\frac{(aq;q^2)_k (q/a;q^2)_k (q;q^2)_k (q^2/b;q^4)_k b^kq^{k^2}}
{(aq^4;q^4)_k(q^4/a;q^4)_k(q^4;q^4)_k (bq;q^2)_k}
\equiv \frac{(q;q^2)_{(n-1)/2}(q^{n+2};q^2)_{(n-1)/2}[n]} {(aq^4;q^4)_{(n-1)/2}(q^4/a;q^4)_{(n-1)/2}}.    \label{eq:q-Hamme-J2L2-2}
\end{align}
\end{lem}
\begin{proof}
The derivation of \eqref{eq:q-Hamme-J2L2} in \cite{GuoZu} uses the formula \cite[Equation (4.6)]{Ra93}:
\begin{align}
&
\sum_{k=0}^\infty\frac{(a;q)_k (1-aq^{3k})(d;q)_k(q/d;q)_k(b;q^2)_k}
{(q^2;q^2)_k (1-a)(aq^2/d;q^2)_k (adq;q^2)_k (aq/b;q)_k}\frac{a^k q^{k+1\choose 2}}{b^k}
\notag\\ &\qquad
=\frac{(aq;q)_\infty (adq/b;q^2)_\infty (aq^2/bd;q^2)_\infty}
{(aq/b;q)_\infty (aq^2/d;q^2)_\infty (adq;q^2)_\infty}.
\label{quadratic}
\end{align}
Letting $q\mapsto q^2$ and taking $a=q$, $d=aq$ and $b\mapsto q^2/b$, we are led to
\begin{align*}
&\sum_{k=0}^{\infty}[6k+1]\frac{(aq;q^2)_k (q/a;q^2)_k (q;q^2)_k (q^2/b;q^4)_k b^kq^{k^2}}
{(aq^4;q^4)_k(q^4/a;q^4)_k(q^4;q^4)_k (bq;q^2)_k}  \\[5pt]
&\quad=\frac{(q^3;q^2)_\infty (abq^2;q^4)_{\infty}(bq^2/a;q^4)_{\infty}}{(bq;q^2)_\infty(aq^4;q^4)_{\infty} (q^4/a;q^4)_{\infty}}.
\end{align*}
Substituting $b=q^{2n}$ into the above identity, we obtain
\begin{align*}
&\sum_{k=0}^{(n-1)/2}[6k+1]\frac{(aq;q^2)_k (q/a;q^2)_k (q;q^2)_k (q^{2-2n};q^4)_k q^{k^2+nk}}
{(aq^4;q^4)_k(q^4/a;q^4)_k(q^4;q^4)_k (q^{n+1};q^2)_k}
= \frac{(q^3;q^2)_{n-1} } {(aq^4;q^4)_{(n-1)/2}(q^4/a;q^4)_{(n-1)/2}}.
\end{align*}
Namely, the $q$-congruence \eqref{eq:q-Hamme-J2L2-2} holds.
\end{proof}

\begin{proof}[Proof of Theorem {\rm\ref{thm:main-2}}]
Suppose that $n\equiv r\pmod4$ with $r=\pm1$.
By \eqref{eq:q-Hamme-J2L2} and \eqref{eq:q-Hamme-J2L2-2} with $b\mapsto b^2$, we have the following $q$-congruence modulo $[n](1-aq^n)(a-q^n)(b-q^n)$:
\begin{align}
&\sum_{k=0}^{(n-1)/2}[6k+1]\frac{(aq;q^2)_k (q/a;q^2)_k (q;q^2)_k (q^2/b^2;q^4)_k b^{2k}q^{k^2}}
{(aq^4;q^4)_k(q^4/a;q^4)_k(q^4;q^4)_k (b^2q;q^2)_k}   \notag\\[5pt]
&\quad\equiv\frac{(q^{2+r}/b;q^4)_{(n-r)/4}}{(bq^{2+r};q^4)_{(n-r)/4}} b^{(n-r)/4}(-q)^{(1-n)/2}
\frac{(b-q^n)(ab-1-a^2+aq^n)}{(a-b)(1-ab)}[n] \notag\\[5pt]
&\quad\quad+\frac{(q;q^2)_{(n-1)/2}(q^{n+2};q^2)_{(n-1)/2}} {(aq^4;q^4)_{(n-1)/2}(q^4/a;q^4)_{(n-1)/2}}
\frac{(1-aq^n)(a-q^n)}{(a-b)(1-ab)}[n], \label{eq:abqn-new}
\end{align}
where we have used the $q$-congruences \eqref{eq:ab-1} and \eqref{eq:ab-2}. It is clear that $q^n\equiv 1\pmod{\Phi_n(q)}$.
By \eqref{eq:lem-1} (with $q\mapsto q^2$) and \eqref{eq:q-bino-a} (with $a=1$), we have
\begin{align}
\frac{(q;q^2)_{(n-1)/2}(q^{n+2};q^2)_{(n-1)/2}} {(aq^4;q^4)_{(n-1)/2}(q^4/a;q^4)_{(n-1)/2}}
&\equiv \frac{(q;q^2)_{(n-1)/2}(q^{2};q^2)_{(n-1)/2}} {(aq^4;q^4)_{(n-1)/2}(q^4/a;q^4)_{(n-1)/2}} \notag\\[5pt]
&\equiv (-q)^{(1-n)/2}\frac{n(1-a)a^{(n-1)/2}}{1-a^n} \pmod{\Phi_n(q)}.  \label{eq:an-frac-new}
\end{align}
Like the proof of Theorem \ref{thm:main-1}, the limit of $b-q^n$ as $b\to 1$ has the factor $\Phi_n(q)$ and the factor $(b^2q;q^2)_{(n-1)/2}$
in the denominator of \eqref{eq:abqn-new} as $b\to 1$ is coprime with $\Phi_n(q)$.
Thus, letting $b\to 1$ in \eqref{eq:abqn} and applying \eqref{eq:an-frac-new} and  \eqref{eq:relation}, we conclude that \eqref{eq:main-2}
is true modulo $\Phi_n(q)^2(1-aq^n)(a-q^n)$.
Further, the $q$-congruence \eqref{eq:q-Hamme-J2L2} also holds for $b=1$, i.e.,
the $q$-congruence \eqref{eq:main-2} is true modulo $[n]$. This completes the proof.
\end{proof}

\section{Proof of Theorem \ref{thm:main-3}}

We first give the following $q$-congruence, which follows from the $c\to 0$ case of \cite[Theorem 6.1]{GS} (see also\cite[Conjecture 4.6]{GuoZu}).
\begin{lem}
Let $n$ be a positive odd integer. Then, modulo $[n](1-aq^n)(a-q^n)$,
\begin{align}
\sum_{k=0}^{n-1}[3k+1]\frac{(aq;q^2)_k(q/a;q^2)_k(q;q^2)_k(q/b;q)_k b^kq^{-{k+1\choose 2}}}
{(aq;q)_k(q/a;q)_k(q;q)_k(bq^2;q^2)_k}
\equiv\frac{(b/q)^{(n-1)/2}(q^2/b;q^2)_{(n-1)/2} }
{(bq^2;q^2)_{(n-1)/2}}[n].
\label{eq:conj-3k+1}
\end{align}
\end{lem}

We also have a simple $q$-congruence for the left-hand side of \eqref{eq:conj-3k+1} modulo $b-q^{n}$.
\begin{lem}
Let $n$ be a positive odd integer. Then, modulo $b-q^n$,
\begin{align}
\sum_{k=0}^{n-1}[3k+1]\frac{(aq;q^2)_k(q/a;q^2)_k(q;q^2)_k(q/b;q)_k b^kq^{-{k+1\choose 2}}}
{(aq;q)_k(q/a;q)_k(q;q)_k(bq^2;q^2)_k}
\equiv
\frac{(q;q^2)_{(n-1)/2}^2[n]} {(aq^2;q^2)_{(n-1)/2}(q^2/a;q^2)_{(n-1)/2}}.
\label{eq:conj-3k+1-2}
\end{align}
\end{lem}
\begin{proof}Using the transformation formula \cite[Equation (3.8.13)]{GR} and the Pfaff--Saalsch\"utz theorem \cite[Appendix (II.12)]{GR},
in the sketch of proof of \cite[Theorem 4.8]{GuoZu} the author and Zudilin gave
\begin{align*}
&\sum_{k=0}^{2N}[3k+1]\frac{(aq;q^2)_k(q/a;q^2)_k(q;q^2)_k(q/b;q)_k(q^{-2N};q)_k(bq^{1+2N};q)_k\,q^k}
{(aq;q)_k(q/a;q)_k(q;q)_k(bq^2;q^2)_k(q^{3+2N};q^2)_k(q^{2-2N}/b;q^2)_k} \\[5pt]
&\quad =\frac{(q^{1-2N};q^2)_\infty(q^2/b;q^2)_\infty(bq^{2+2N};q^2)_\infty}
{(1-q)\,(q^{3+2N};q^2)_\infty(q^{2-2N}/b;q^2)_\infty(bq^2;q^2)_\infty}
\frac{(abq;q^2)_N(aq^{1-2N}/b;q^2)_N}{(aq^2;q^2)_N(aq^{-2N};q^2)_N} \\[5pt]
&\quad =\frac{(q;q^2)_N^2(abq;q^2)_N(bq/a;q^2)_N [2N+1]}
{(b;q^2)_N(bq^2;q^2)_N(aq^2;q^2)_N(q^2/a;q^2)_N}
\end{align*}
(we correct a typo in the first equality here). Letting $N=(n-1)/2$ and $b\to 0$ in the above identity, we are led to
\begin{align*}
\sum_{k=0}^{n-1}[3k+1]\frac{(aq;q^2)_k(q/a;q^2)_k(q;q^2)_k(q^{1-n};q)_k q^{nk-{k+1\choose 2}}}
{(aq;q)_k(q/a;q)_k(q;q)_k(q^{n+2};q^2)_k}
=\frac{(q;q^2)_{(n-1)/2}^2[n]} {(aq^2;q^2)_{(n-1)/2}(q^2/a;q^2)_{(n-1)/2}}.
\end{align*}
Namely, the desired $q$-congruence holds.
\end{proof}

\begin{proof}[Proof of Theorem {\rm\ref{thm:main-3}}]
From \eqref{eq:conj-3k+1} and \eqref{eq:conj-3k+1-2} we can deduce that, modulo $[n](1-aq^n)(a-q^n)(b-q^n)$,
\begin{align}
&\sum_{k=0}^{n-1}[3k+1]\frac{(aq;q^2)_k(q/a;q^2)_k(q;q^2)_k(q/b;q)_k b^kq^{-{k+1\choose 2}}}
{(aq;q)_k(q/a;q)_k(q;q)_k(bq^2;q^2)_k}  \notag\\[5pt]
&\quad\equiv\frac{(b/q)^{(n-1)/2}(q^2/b;q^2)_{(n-1)/2} }{(bq^2;q^2)_{(n-1)/2}}
\frac{(b-q^n)(ab-1-a^2+aq^n)}{(a-b)(1-ab)}[n] \notag\\[5pt]
&\quad\quad+\frac{(q;q^2)_{(n-1)/2}^2} {(aq^2;q^2)_{(n-1)/2}(q^2/a;q^2)_{(n-1)/2}}
\frac{(1-aq^n)(a-q^n)}{(a-b)(1-ab)}[n], \label{eq:3k+1-new}
\end{align}
where we have utilized \eqref{eq:ab-1} and \eqref{eq:ab-2}. Note that the right-hand sides of \eqref{eq:abqn}
and \eqref{eq:3k+1-new} are exactly the same. Thus, letting $b\to 1$ in \eqref{eq:3k+1-new}, we arrive at
\eqref{eq:main-3}.
\end{proof}

\section{Concluding remarks and open problems}

We first give the following $q$-congruence related to Theorem \ref{thm:main-1}. To the best of our knowledge, this is
the first $q$-congruence modulo $[n]\Phi_n(q)(1-aq^n)(a-q^n)(b-q^n)$ in the literature.
\begin{thm}
Let $n$ be a positive odd integer. Then, modulo $[n]\Phi_n(q)(1-aq^n)(a-q^n)(b-q^n)$,
\begin{align}
&\sum_{k=0}^{(n-1)/2}[4k+1]\frac{(aq;q^2)_k (q/a;q^2)_k (q/b;q^2)_k (q;q^2)_k}
{(aq^2;q^2)_k(q^2/a;q^2)_k (bq^2;q^2)_k (q^2;q^2)_k}b^k \notag\\[5pt]
&\quad\equiv \sum_{k=0}^{n-1}[4k+1]\frac{(aq;q^2)_k (q/a;q^2)_k (q/b;q^2)_k (q;q^2)_k}
{(aq^2;q^2)_k(q^2/a;q^2)_k (bq^2;q^2)_k (q^2;q^2)_k}b^k. \label{eq:equiv}
\end{align}
In particular, we have
\begin{align}
\sum_{k=0}^{(n-1)/2}[4k+1]\frac{(q;q^2)_k^4}{(q^2;q^2)_k^4}
\equiv \sum_{k=0}^{n-1}[4k+1]\frac{(q;q^2)_k^4}{(q^2;q^2)_k^4} \pmod{[n]\Phi_n(q)^4}.  \label{eq:equiv-2}
\end{align}
\end{thm}
\begin{proof}
By \cite[Appendix (I.11)]{GR}, we have
$$
\frac{(a;q)_{n-k}}{(b;q)_{n-k}}=\frac{(a;q)_n (q^{1-n}/b;q)_k}{(b;q)_n (q^{1-n}/a;q)_k}\left(\frac{b}{a}\right)^k
\equiv\frac{(a;q)_n (q/b;q)_k}{(b;q)_n (q/a;q)_k}\left(\frac{b}{a}\right)^k \pmod{\Phi_n(q)}.
$$
It follows that
\begin{align}
&\sum_{k=(n+1)/2}^{n-1}[4k+1]\frac{(aq;q^2)_k (q/a;q^2)_k (q/b;q^2)_k (q;q^2)_k}
{(aq^2;q^2)_k(q^2/a;q^2)_k (bq^2;q^2)_k (q^2;q^2)_k}b^k  \notag \\[5pt]
&\quad=\sum_{k=1}^{(n-1)/2}[4(n-k)+1]\frac{(aq;q^2)_{n-k} (q/a;q^2)_{n-k} (q/b;q^2)_{n-k} (q;q^2)_{n-k}}
{(aq^2;q^2)_{n-k}(q^2/a;q^2)_{n-k} (bq^2;q^2)_{n-k} (q^2;q^2)_{n-k}}b^{n-k} \notag\\[5pt]
&\quad\equiv \frac{(aq;q^2)_n (q/a;q^2)_n (q/b;q^2)_n (q;q^2)_n b^n}
{(aq^2;q^2)_n(q^2/a;q^2)_n (bq^2;q^2)_n (q^2;q^2)_{n-1}}  \notag\\[5pt]
&\qquad\times\sum_{k=1}^{(n-1)/2}[1-4k]\frac{(1/a;q^2)_{k} (a;q^2)_{k} (1/b;q^2)_{k} (q^{2};q^2)_{k-1}}
{(q/a;q^2)_{k}(aq;q^2)_{k} (bq;q^2)_{k} (q;q^2)_{k}}b^kq^{4k}  \pmod{\Phi_n(q)^2}, \label{eq:sum}
\end{align}
where we have used the fact $(q;q^2)_n\equiv 0\pmod{\Phi_n(q)}$. Like \cite[Lemma 3.1]{GS}, we can show that
\begin{align*}
\frac{(a;q^2)_{(n+1)/2-k}}{(q/a;q^2)_{(n+1)/2-k}}
&\equiv (-a)^{(n+1)/2-2k}\frac{(a;q^2)_k}{(q/a;q^2)_k} q^{(n^2-1)/4+k} \pmod{\Phi_n(q)},\\[5pt]
\frac{(q^2;q^2)_{(n+1)/2-k-1}}{(q;q^2)_{(n+1)/2-k}}
&\equiv (-1)^{(n+1)/2}\frac{(q^2;q^2)_{k-1}}{(q;q^2)_k} q^{(n^2-1)/4+k}
\end{align*}
for $1\leqslant k\leqslant (n-1)/2$, and so the $k$-th and $(n+1)/2-k$-th terms in the summation of the right-hand
side of  \eqref{eq:sum} cancel each other modulo $\Phi_n(q)$.
Noticing that the fraction before the summation is congruent to $0$ modulo $\Phi_n(q)$ too,
we conclude that the right-hand side of  \eqref{eq:sum} is congruent to $0$ modulo $\Phi_n(q)^2$.
This proves \eqref{eq:equiv} modulo $\Phi_n(q)^2$. By the proof of Theorem \ref{thm:main-1}, we know that
\eqref{eq:equiv} also holds modulo $[n](1-aq^n)(a-q^n)(b-q^n)$. This completes the proof.
\end{proof}

Letting $n=p^r$ be an odd prime power and letting $q\to 1$ in \eqref{eq:equiv-2}, we obtain
\begin{align}
\sum_{k=0}^{(p^r-1)/2}(4k+1)\frac{(\frac{1}{2})_k^4}{k!^4}
\equiv \sum_{k=0}^{p^r-1}(4k+1)\frac{(\frac{1}{2})_k^4}{k!^4} \pmod{p^{r+4}}.
\end{align}
Recall that the {\it Bernoulli numbers} $B_n$ are defined as follows:
$$
B_0=1,\quad \sum_{k=0}^{n}{n+1\choose k}B_k=0,\quad\text{for}\ n=1,2,\ldots.
$$
Based on numerical calculations, we would like to propose the following conjecture.
\begin{conj}Let $p>3$ be a prime and $r$ a positive integer. Then
\begin{align}
\sum_{k=0}^{(p^r-1)/2}(4k+1)\frac{(\frac{1}{2})_k^4}{k!^4}
\equiv p^r+\frac{7}{6}B_{p-3}p^{r+3} \pmod{p^{r+4}}.  \label{eq:Bernoulli}
\end{align}
\end{conj}

Note that Sun \cite[Conjecture 5.1(ii)]{Sun} conjectured that, for any prime $p>3$ and integer $r\geqslant 1$,
\begin{align}
\sum_{k=0}^{p^r-1} (3k+1)\frac{(\frac{1}{2})_k^3}{k!^3}2^{2k}
\equiv p^r+\frac{7}{6}B_{p-3}p^{r+3} \pmod{p^{r+4}}. \label{eq:Sun}
\end{align}
Since the right-hand sides of \eqref{eq:Bernoulli} and \eqref{eq:Sun} are exactly the same,
it is natural to raise the following new conjecture (which is also valid for $p=3$).

\begin{conj}Let $p$ be an odd prime and $r$ a positive integer. Then
\begin{align}
\sum_{k=0}^{(p^r-1)/2}(4k+1)\frac{(\frac{1}{2})_k^4}{k!^4}
\equiv \sum_{k=0}^{p^r-1} (3k+1) \frac{(\frac{1}{2})_k^3}{k!^3}2^{2k}  \pmod{p^{r+4}}.  \label{eq:a=b}
\end{align}
\end{conj}

Although it is very difficult to give $q$-analogues of supercongruences involving Bernoulli numbers,
we find a $q$-analogue of \eqref{eq:a=b}, which is included in the following conjecture.
\begin{conj}\label{conj:final}
Let $n$ be a positive odd integer. Then, modulo $[n]\Phi_n(q)(1-aq^n)(a-q^n)(b-q^n)$,
\begin{align}
&\sum_{k=0}^{(n-1)/2}[4k+1]\frac{(aq;q^2)_k (q/a;q^2)_k (q/b;q^2)_k (q;q^2)_k}
{(aq^2;q^2)_k(q^2/a;q^2)_k (bq^2;q^2)_k (q^2;q^2)_k}b^k \notag\\[5pt]
&\quad\equiv \sum_{k=0}^{n-1}[3k+1]\frac{(aq;q^2)_k(q/a;q^2)_k(q;q^2)_k(q/b;q)_k b^kq^{-{k+1\choose 2}}}
{(aq;q)_k(q/a;q)_k(q;q)_k(bq^2;q^2)_k}.    \label{eq-conj:final}
\end{align}
In particular, we have
\begin{align}
\sum_{k=0}^{(n-1)/2}[4k+1]\frac{(q;q^2)_k^4}{(q^2;q^2)_k^4}
\equiv \sum_{k=0}^{n-1}[3k+1]\frac{(q;q^2)_k^3 q^{-{k+1\choose 2} } }{(q;q)_k^2 (q^2;q^2)_k} \pmod{[n]\Phi_n(q)^4}. \label{eq:a=b-q}
\end{align}
\end{conj}

It is easy to see that \eqref{eq:a=b} follows from \eqref{eq:a=b-q} by taking $n=p^r$ and $q\to 1$.
Note that the $q$-congruence \eqref{eq-conj:final} is true modulo $[n](1-aq^n)(a-q^n)(b-q^n)$  by \eqref{eq:abqn} and \eqref{eq:3k+1-new}.
Thus, to prove Conjecture \ref{conj:final} it suffices to show that \eqref{eq-conj:final} is true modulo $\Phi_n(q)^2$.

\vskip 5mm \noindent{\bf Acknowledgment.}
This work was partially supported by the National Natural Science Foundation of China (grant 11771175).

\end{document}